\documentclass[a4paper,11pt]{amsart}

\usepackage{header}

\title[Rationalizability of field extensions]{Rationalizability of field extensions\\with a view towards Feynman integrals}

\author{Dino Festi, Andreas Hochenegger}

\begin{document}
\maketitle

\begin{abstract}
In 2021, Marco Besier and the first author introduced the concept of rationalizability of square roots to simplify arguments of Feynman integrals.
In this work, we generalize the definition of rationalizability to field extensions.
We then show that the rationalizability of a set of quadratic field extensions is equivalent to the rationalizability of the compositum of the field extensions,
providing a new strategy to prove rationalizability of sets of square roots of polynomials.
\end{abstract}

\section{Introduction}

Feynman integrals play an important role in high-energy physics, see for example \cite{Weinzierl} for a comprehensive introduction.
From a mathematical point of view, a Feynman integral is interesting, as the argument often contains a square root.
This square root makes the numerical evaluation of the integral computationally more expensive, preventing the achievement of higher precision that could lead to more incisive experiments.
This issue can often be overcome with a change of variables in the integrals making the square root disappear and leaving only a rational expression.
Unfortunately, such change of variables does not always exist and stating its (non-)existence can be a cumbersome problem for physicists with no strong background in algebraic geometry as well as for experienced geometers. 
This challenge led Marco Besier and the first author to introduce the concept of \emph{rationalizability} of square roots of polynomials, in~\cite{BF21}.
Roughly speaking, a square root of a polynomial is rationalizable if there is a rational change of variables making the square root disappear, i.e., turning the polynomial into a square.
In the same paper, some practical criteria for the rationalizability of square roots of polynomials in one and two variables are given.
Finally, the concept of rationalizability for \emph{sets} (also called \emph{alphabets} in that paper) of square roots of polynomials is introduced. 
On the one hand,  there is a practical strategy to disprove rationalizability of a given set (cf.~\cite[Proposition 47 and Remark 48]{BF21}), 
but on the other hand the original paper does not offer a solid strategy to \emph{prove} the rationalizability (cf.~\cite[Remarks 48 and 52]{BF21}).

In this paper, we try to fill in this gap.
To do so, we introduce the notion of rationalizability for field extensions, generalizing the concept of rationalizability for square roots.
\begin{definition}
Let $K$ be any field.
We say that a field extension $K\subset L$ is \emph{rationalizable} if there exists a non-zero homomorphism $\psi\colon L\to K$.

For $i\in \{ 1,\ldots,m\}$ let $L_i$ be a field extension of $K$, with $\iota_i\colon K\into L$ the natural inclusion.
We say that the set $\{ L_1,\ldots,L_m \}$ is \emph{rationalizable} if for each $i$ there is a non-zero homomorphism $\psi_i\colon L_i \to K$ such that
$$
\psi_1 \circ \iota_1 = \cdots= \psi_m \circ \iota_m\; .
$$
\end{definition}

\begin{example}
Consider $K=\bC (x)$, $p=z^2-(x-1)\in K[z]$ and 
$$
L\coloneqq K(\sqrt{x-1}) \cong \frac{K[z]}{(p)}.
$$
Then the inclusion $\iota \colon K \into L$ is rationalizable.
Indeed consider the map $\psi\colon L \to K$ defined as the identity on the elements of $\bC$ and by $x\mapsto x^2+1, z\mapsto x$.
With this definition, the map $\phi \coloneqq  \psi \circ \iota \colon K \to K$ sends $x-1$ to $x^2$.
\end{example}

\begin{example}
Consider $K=\bC (x)$ and set $L_1\coloneqq K (\sqrt{x}), L_2\coloneqq K(\sqrt{1-x})$.
Then the set of field extensions $\{ L_1, L_2\}$ is rationalizable.
Indeed, for $i=1,2$, consider the  homomorphisms
$\psi_i \colon L_i \to K$ defined as follows.
\begin{align*}
    \psi_1\colon & \hspace{10pt} \sqrt{x} \hspace{11pt} \mapsto \hspace{3pt} \frac{2x^2}{1+x^2}\\
    \psi_2 \colon & \sqrt{1-x} \hspace{3pt}\mapsto \hspace{3pt} \frac{1-x^2}{1+x^2}
\end{align*}

One can easily check that
$$
\psi_1 \circ \iota_1 = \psi_2 \circ \iota_2 \eqqcolon \phi\; ,
$$
where $\phi\colon K \to K$ is the homomorphism defined by 
$$
\phi\colon x \mapsto \left( \frac{2x^2}{1+x^2}\right)^2\; .
$$
Later we will see that the rationalizability of the set considered in this example does not depend on the special form of the polynomials but only on their degree, cf. \autoref{c:TwoLinear}.
\end{example}

\begin{remark}
The rationalizing map at the end of \cite[Remark 52]{BF21} is wrong.
The map $\sigma$ should send $X$ to $\frac{2X}{X^2+1}$.
Consequently, the composition $\sigma \circ \psi$ should send $X$ to $\left( \frac{2X}{X^2+1}\right)^2+1$. 
\end{remark}

Our main result is the following theorem, which shows that the rationalizability of a set of quadratic field extensions is equivalent to the rationalizability of the compositum of the field extensions.
Here and in the rest of the paper $\overline{K}$ will always denote a fixed algebraic closure of a field $K$.
\begin{theorem}\label{t:split}
Let $K$ be a field 
 $\alpha_1,\ldots,\alpha_m$ be  elements of $\overline{K}$ 
 such that all the extensions $L_i\coloneqq K(\alpha_i)$ are of degree two over $K$.
Then the set of extensions  $\{L_1,\ldots,L_m\}$ is rationalizable if and only if $K\subset K(\alpha_1,\ldots,\alpha_m)$ is.
\end{theorem}
This allows us to develop a new strategy to prove rationalizability of sets of square roots of polynomials (cf.~\autoref{r:Strategy}). 

We proceed as follows.
In~\autoref{s:Preliminaries} we recall the notion of rationalizability for square roots of polynomials and some well-known facts about field extensions.
In~\autoref{s:Rationalizability} 
we  prove  \autoref{t:split} 
and we outline a strategy to prove rationalizability of sets of square roots.
Finally, in~\autoref{s:Applications} we apply \autoref{t:split} to some sets of square roots of polynomials.
The results thus obtained, joint with~\cite[Proposition 47]{BF21}, 
led us to formulate the following conjecture.
\begin{conjecture}\label{conj:Sets}
The set of square roots $\left\{ \sqrt{f_1},\ldots,\sqrt{f_m} \right\}$ is rationalizable if and only if for every non-empty subset $J\subseteq \{1,\ldots,m\}$ the square root 
$$
\sqrt{\prod_{j\in J} f_j}
$$ 
is rationalizable. 
\end{conjecture}

\subsection*{Acknowledgments} We thank Marco Besier for stimulating conversations and useful remarks. 
We also thank the anonymous referee for reading carefully through the text and their useful suggestions.

\section{Preliminaries}\label{s:Preliminaries}

In this section we recall the notion of rationalizability of square roots introduced in~\cite{BF21} and we extend it to higher order roots;
we also recall some basic results in field theory.
Proofs are omitted or referred to the original sources.

\subsection{Rationalizability of square roots}

Let $\kk$ be any field and consider the $\kk$-algebra $Q\coloneqq \kk(x_1,\ldots,x_n)$.

\begin{definition}
Consider $q \in Q$.
We say that the square root $\sqrt{q}$ is \emph{rationalizable} if there exists $\phi \colon Q \to Q$ such that $\phi(q)$ is a square in $Q$.

We say that a family of square roots $\{\sqrt{q_1},\ldots,\sqrt{q_m}\}$ is \emph{rationalizable} if there exists $\phi \colon Q \to Q$ such that $\phi(q_i)$ is a square in $Q$ for $i \in \{1,\ldots,m\}$.
\end{definition}

We recall the main theorems of~\cite{BF21} as they will serve as a starting point for the theory developed here.
The notation and statements are slightly changed to better fit the content of this article.

\begin{theorem}[{\cite[Theorem 1]{BF21}}]
The following statements hold:
\begin{enumerate}
    \item Given $q \in Q = \kk(x_1,\ldots,x_n)$, there exists a squarefree polynomial $f\in R = \kk[x_1,\ldots,x_n]$ such that $\sqrt{q}$ is rationalizable if and only if $\sqrt{f}$ is.
    \item For a squarefree polynomial $f \in R$, the square root $\sqrt{f}$ is rationalizable if and only if the variety $V(z^2-f)\subset \Spec R[z] = \bA^{n+1}_\kk$ is unirational.
    \item Assume $\kk$ is algebraically closed and let $d$ denote the degree of a squarefree polynomial $f \in R$.   
    If $d=1,2$, or if $\overline{V}$ has a singular point of multiplicity $d-1$, then $\sqrt{f}$ is rationalizable.
\end{enumerate}
\end{theorem}

\begin{corollary}
Consider a squarefree polynomial $f \in \kk[x]$.
Then the square root of $f$ is rationalizable if and only if $V(z^2-f) \subset \bA^2$ has geometric genus zero.
\end{corollary}

\subsection{Rationalizability of roots of higher order}
One can naturally generalize the concept of rationalizability of square roots to roots of higher order.
Although it is a very natural step, this subsection will not be needed in the rest of the paper and non-interested reader can safely skip it.
Also in this subsection, let $\kk$ be any field and $Q$ the $\kk$-algebra $\kk(x_1,\ldots,x_n)$.
Consider $q\in Q$ of degree $d$ and an integer $e\geq 2$. 

\begin{definition}
 We say that the root $\sqrt[e]{q}$ is \emph{rationalizable} if there exists a homomorphism of $\kk$-algebras $\phi\colon Q\to Q$ such that $\phi (q)$ is an $e$-power in $Q$, that is, $\phi(q)=h^e$, for some element $h\in Q$.
We say that a family of roots $\{\sqrt[e_1]{q_1},\ldots, \sqrt[e_m]{q_m} \}$ is rationalizable if there exists a homomorphism of $\kk$-algebras $\phi\colon Q\to Q$ such that for every $i=1,\ldots,m$ we have $\phi (q_i)=h_i^{e_i}$ for some $h_i\in Q$
\end{definition}

Working exactly as in~\cite[\S 2]{BF21}, one can prove the following result.

\begin{proposition}\label{p:HigherOrder}
The following statements hold:
\begin{enumerate}
    \item Given $q \in Q = \kk(x_1,\ldots,x_n)$, there exists a polynomial $f\in R = \kk[x_1,\ldots,x_n]$ such that its factorization does not contain any $e$-th power and  $\sqrt[e]{q}$ is rationalizable if and only if $\sqrt[e]{f}$ is.
    \item For $f\in R$ whose factorization does not contain any $e$-th power, the root $\sqrt[e]{f}$ is rationalizable if and only if the variety $V(z^e-f)\subset \Spec R[z] = \bA^{n+1}_\kk$ is unirational.
\end{enumerate}
\end{proposition}

\begin{remark}\label{r:GenusHigherOrder}
Consider $f\in\bC [x]$ such that its factorization does not contain any $e$-th power.
Assume that the curve $C=V(z^e-f)$ is irreducible. 

Then following \cite[\S 2]{Kon96}, we may assume that $C$ is given by the equation
\begin{equation}
\label{eq:C}
C\colon z^e = x^{\ell_0} \cdot \prod_{i=1}^m (x-a_i)^{\ell_i}
\end{equation}
where $\sum_{i=0}^m \ell_i \equiv 0 \bmod e$.

To see this, write $f = \prod_{i=1}^m (x-\alpha_i)^{\ell_i}$ and let $d = \sum_{i=1}^m \ell_i$.
Without loss of generality, we may assume that all $\alpha_i$ are non-zero.
Consider the weighted homogenization of $z^e-f$ with respect to $y$, where $\deg(z) = \lceil \frac d e \rceil$, $\deg(x)=1 = \deg(y)$. 
By dehomogenizing the new equation with respect to $x$, we get equation \eqref{eq:C} with $a_i = \frac 1 {\alpha_i}$.

Moreover, we have again by  \cite[\S 2]{Kon96}
that its geometric genus is
$$
g(C)=\frac{(e-1)(s-2)}{2},
$$
where
$$
s=\begin{cases}
m &\textrm{if } m\equiv 0 \bmod e\; ,\\
m+1 &\textrm{otherwise}\; ,
\end{cases}
$$
and $m$ is defined as in equation \eqref{eq:C}.

In view of~\autoref{p:HigherOrder}, this yields that 
the root $\sqrt[e]{f}$ is rationalizable if and only if $s=2$,
i.e., $m=e=2$ or $m=1$.
\end{remark}

\subsection{Field theory}

Before a further generalization of the notion of rationalizability,  
we recall a few well-known facts about algebraic field extensions. 
For more information and details see, for example, \cite[Chapter 3]{Bos18}.

\begin{definition}
The field extension $K\subset L$ is called \emph{separable} if for every element $\alpha\in L$, the minimal polynomial of $\alpha$ over $K$ is separable, i.e., it has no repeated roots in any field extension of $K$.
\end{definition}

\begin{proposition}[{\cite[Proposition 3.6.2]{Bos18}}]
If $K$ is a field of characteristic $0$, then every algebraic extension is separable.
\end{proposition}

\begin{proposition}[Primitive element theorem]
Let $K \subset L$ be a separable field extension. Then there is an irreducible element $p \in K[z]$ such that $L \cong K[z]/(p)$.
\end{proposition}

For a proof of this result  we refer to \cite[Proposition 3.6.12]{Bos18}.
In the examples in this article we will only consider fields of characteristic $0$ and hence separable extensions. 
Therefore, using the above proposition, 
all the considered field extensions $K \subset L$ will always be of the form $L = K[z]/(p)$ with an irreducible polynomial $p$. 

\begin{remark}\label{r:PrimitiveElement}
Notice that it is also possible to write $L = K(\alpha)$ where $\alpha \in L$ is a root of $p$ (cf.~\cite[Proposition 3.2.6]{Bos18}). The element $\alpha$ (or $p$) is called a \emph{primitive element} of the extension $K \subset L$.
\end{remark}

\begin{example}
Let $f \in K$ be a squarefree element.
Then the polynomial $z^2-f \in K[z]$ is irreducible and $L = K[z]/(z^2-f)$ is a field extension of $K$.
We can also write $L = K(\sqrt f)$, where $\sqrt f$ denotes a square root of $f$.
\end{example}

\section{Rationalizability of field extensions}\label{s:Rationalizability}
 
In this section we prove the main theorem of the paper, \autoref{t:split}, 
and we outline how this theorem can be used to prove rationalizability of sets of square roots.

For the whole section, let $\kk$ be an algebraically closed field of characteristic $0$ and set 
$R \coloneqq  \kk[x_1,\ldots,x_n]$ and $Q \coloneqq  \Frac(R) = \kk(x_1,\ldots,x_n)$,
with $n\geq 1$; fix an algebraic closure $\overline{Q}$ of $Q$;
as before, $K$ will denote any field and $K\subset L$ a field extension;
the embedding $K\into L$ will be denoted by $\iota$;
we fix an algebraic closure $\overline{K}$ of $K$.

Let us recall the definition of rationalizability for field extensions given in the introduction.

\begin{definition}
We say that a field extension $K \subset L$ is \emph{rationalizable} if and only if there is a non-zero homomorphism $\psi \colon L \to K$.
\end{definition}

\begin{remark}
By taking the composition $\phi = \psi \circ \iota$, we get a commutative diagram
\[
\begin{tikzcd}
K \ar[rr, "\phi"] \ar[dr, hook, "\iota"'] && K \\
& L \ar[ur, "\psi"']
\end{tikzcd}
\]
Note that both $\phi$ and $\psi$ are automatically injective.
\end{remark}

\begin{definition}
Let $K$ be a field and let $\iota_i \colon K \into L_i$ be field extensions for $i \in \{1,\ldots, m\}$.
We say that the set of field extensions $\{L_1,\ldots,L_m\}$ is rationalizable if for each $i=1,\ldots,m$ the field extension $K\subset L_i$ is rationalizable and there exists a rationalizing homomorphism $\psi_i\colon L_i \to K$ such that 
$$
\psi_1\circ \iota_1 = \cdots = \psi_m \circ \iota_m \eqqcolon \phi\; .
$$
In other words, for each $i \in \{1,\ldots, m\}$ we have the following commutative diagram:
\[
\begin{tikzcd}
K \ar[rr, "\phi"] \ar[dr, hook, "\iota_i"'] && K \\
& L_i \ar[ur, "\psi_i"'] 
\end{tikzcd}
\]

\end{definition}

\begin{remark}
The central point of the definition is, that we can choose the homomorphisms $\psi_i \colon L_i \to K$ in such a way, that we get the same homomorphism $\phi = \psi_i \circ \iota_i$, independently of $i \in \{1,\ldots,m\}$.
\end{remark}

\subsection{First properties}
In this subsection  we show that the rationalizability of a square root in the sense of \cite{BF21} is indeed a special case of the rationalizability of a field extension.

\begin{proposition}
\label{p:EquivSingle}
Consider $f \in R = \kk[x_1,\ldots,x_n]$ squarefree.
Then $\sqrt{f}$ is rationalizable if and only if the field extension $Q \subset Q(\sqrt{f})\cong Q[z]/(z^2-f)$ is rationalizable.
\end{proposition}

Note that $f$ being squarefree implies the irreducibility of $z^2-f$,
assuring that  $z^2-f$ is the minimal polynomial of $\sqrt{f}$ and hence that $Q[z]/(z^2-f)$ is indeed isomorphic to $Q(\sqrt{f})$.

\begin{proof}
Assume  $\sqrt{f}$ to be rationalizable: there is a  $\phi \colon Q \to Q$ with $\phi(f) = h^2$, $h\in Q$.
Then we define
$\psi \colon Q[z]/(z^2-f) \to Q$ by mapping $q \in Q$ to $\psi(q) \coloneqq \phi(q)$ and by setting $\psi(z) \coloneqq h$.
This is indeed a well-defined homomorphism, as $\psi( z^2 - f ) = (\psi(z))^2 - \psi(f) = 0$.
Since $\psi$ is non-zero, we conclude that $Q \subset Q[z]/(z^2-f)$ is rationalizable.

Conversely, assume $Q \subset Q[z]/(z^2-f)$ to be rationalizable, so there is a non-zero homomorphism $\psi \colon Q[z]/(z^2-f)$.
Using the composition 
$$
\phi \coloneqq \psi \circ \iota \colon Q \into Q[z]/(z^2-f) \to Q\; ,
$$ 
we see that $\phi(f) = \psi(z)^2$ is a square.
Hence the square root of $f$ is rationalizable.
\end{proof}

The following lemma is a generalization of \cite[Lem. 2.7 \& Cor. 2.8]{BF21}.

\begin{lemma}\label{l:polynomial}
Let $L$ be an extension of $Q$ and write $L \cong Q[z]/(q)$ for some irreducible $q \in Q[z]$. 
Then there is an irreducible $p \in R[z]$ such that $L \cong Q[z]/(p)$.

If $q$ is monic, then $p$ can also be chosen to be monic.
\end{lemma}

\begin{proof}
As $R$ is a unique factorization domain, there is a minimal $u \in R$ such that $p \coloneqq u\cdot q \in R[z]$, by clearing denominators.
Moreover, $p$ is irreducible in $R[z]$, as any non-trivial factorization of $p$ would give rise to a factorization of $q$.
As $u$ is a unit in $Q$, we get that $Q' \cong Q[z]/(p)$.

In the case that $q$ is monic, that is, $q = z^n + c_{n-1} z^{n-1} + \cdots + c_0$, we get that $p = uq = uz^n + uc_{n-1}z^{n-1} + \cdots + uc_0$ is not monic anymore.
By defining the monic $\tilde p = z^n + \tilde c_{n-1} z^{n-1} + \cdots + \tilde c_0$ with $\tilde c_i = c_i \cdot u^{n-i}$, we get that $\tilde p(uz) = u^{n} p(z)$.
As $Q[z]/(\tilde p) \to Q[z]/(p), z \mapsto uz$ is an isomorphism, we find that $L \cong Q[z]/(\tilde p)$ with $\tilde p \in R[z]$ monic.
\end{proof}

Note that the above proof uses only that $R$ is a unique factorization domain.

We give a geometric interpretation of rationalizability of a field extension in the spirit of \cite[\S 2.2]{BF21}.

\begin{proposition}\label{p:alggeom}
Let $L$ be an extension of $Q$.
Then there exists an irreducible polynomial $p\in R[z]$ 
such that
$Q \subset L$ is rationalizable if and only if $V(p) \subset \Spec R[z]$ is unirational.
\end{proposition}

\begin{proof}
Recall that $\Spec R = \bA^n$ is rational, as $Q = \kk(x_1,\ldots,x_m)$.

From \autoref{l:polynomial} then there exists a polynomial $p\in R[z]$ such that $L\cong Q[z]/(p)$.

If $Q \subset L\cong  Q[z]/(p)$ is rationalizable, we have an (injective) homomorphism $\psi \colon  Q[z]/(p) \to Q$.
Note that $Q[z]/(p) = \Frac(R[z]/(p)) = k (V(p))$, the function field of $V(p)$, so 
\[
\iota \circ \psi \colon \Frac(R[z]/(p)) \into \kk(x_1,\ldots,x_m)
\]
implies that $V(p)$ is unirational.

Conversely, 
suppose that $V(p)$ is unirational, that is, there is an injective homomorphism
$$
k(V(p)) = Q[z]/(p) = \Frac(R[z]/(p)) \into \kk(x_1,\ldots,x_m) \cong Q.
$$
This shows that $Q \subset Q[z]/(p)$ is rationalizable.
\end{proof}

\begin{remark}
The two implications can be generalized by relaxing some hypotheses on $R$.

The first implication, i.e., that if 
$Q \subset Q[z]/(p)$ is rationalizable then $V(p)$ is unirational, holds also for any $R$ such that $\Spec R$ is unirational.

The converse implication holds also for an $R$ such that   $\kk(x_1,\ldots,x_m) \into Q=\Frac (R)$.

In particular, the statement of \autoref{p:alggeom} is true for any $Q = \Frac (R)$ where $R$ is such that $\Spec R$ is rational.
\end{remark}

Finally, also the rationalizability of a family $\{\sqrt{f_1},\ldots,\sqrt{f_m}\}$ can be understood in the context of field extensions.
Indeed, by the same arguments as in~\autoref{p:EquivSingle}, one can check that this definition generalizes the rationalizability of a set of square roots.
\begin{proposition}\label{p:EquivFamilies}
Consider $f_1,\ldots,f_m\in R$ squarefree polynomials and 
define the field extensions 
$$
Q\subset L_i\coloneqq  Q (\sqrt{f_i}) \cong \frac{Q[z]}{(z^2-f_i)}.
$$
Then the set of square roots$\{\sqrt{f_1},\ldots,\sqrt{f_m}\}$ is rationalizable if and only if the set of field extensions $\{L_1,\ldots,L_m\}$ is.
\end{proposition}
\begin{proof}
As in \autoref{p:EquivSingle}.
\end{proof}

\begin{lemma}\label{l:tower}
Let $K\eqqcolon M_0 \subset M_1 \subset M_2 \subset \cdots \subset M_m \eqqcolon M$ be a tower of extensions
such that $M_{i-1}\subset M_i$ is rationalizable for every $i=1,\ldots,m$.
Then $K\subset M$ is rationalizable.
\end{lemma}
\begin{proof}
For every $i=1,\ldots,m$, let $\psi_i\colon M_i\to M_{i-1}$ denote the rationalizing map.
We want to construct a field homomorphism
$\psi\colon M\to K$.
To do so, simply consider 
\[
\psi\coloneqq \psi_1\circ \cdots \circ \psi_m\colon M_m=M\to M_0=K . \qedhere
\]
\end{proof}

\subsection{The main theorem}
In this subsection we prove the main result of the paper, \autoref{t:split}, and we highlight its impact on the study of rationalizability of families, cf. \autoref{c:FamilySquareRoots} and \autoref{r:Strategy}.

\begin{proposition}\label{p:family}
Let $K$ be a field and  $\alpha_1,\ldots,\alpha_m$ be elements of $\overline{K}$ such that, for $i=2,\ldots,m$, 
  the extensions  $K\subset L_i\coloneqq K(\alpha_i)$ are of degree two.
If the set $\{ L_1,\ldots, L_m \}$ is rationalizable 
then $K\subset K(\alpha_1,\ldots,\alpha_m)$ is rationalizable.
\end{proposition}
\begin{proof}
Let $\phi\colon K\to K$ and $\psi_i \colon L_i \to K$ denote the rationalizing homomorphisms of the set $\{ L_1,\ldots,L_m\}$.
Note that the homomorphisms  $\psi_i$ are completely determined by $\psi_i |_K = \phi$ and the image of $\psi_i(\alpha_i)$.
We inductively construct  a sequence of field extensions 
$$
K \subset M_1 \subset M_2 \subset \cdots \subset M_m
$$ 
that fit into the following commutative diagram.
\[
\begin{tikzcd}[/tikz/cells={/tikz/nodes={shape=asymmetrical
  rectangle,text width=0.75cm,text height=1.5ex,text depth=0.3ex,align=center}}]
K \ar[rrrr, "\phi"] \ar[rd, hook] && & & K \\
& M_1 \ar[rrru, "\pi_1"'] \ar[rd, hook] \\
& & \rotatebox{5}{\raisebox{-0.6ex}{$\ddots$}} \ar[rd, hook] \\
& & & M_m \ar[ruuu, "\pi_m"']
\end{tikzcd}
\]
As starting point we take $M_1=L_1$ and $\pi_1=\psi_1$.
For $i=2,\ldots,m$ define 
$$
M_i\coloneqq M_{i-1}(\alpha_i)= K(\alpha_1,\ldots,\alpha_i).
$$
Note that $M_m=K(\alpha_1,\ldots,\alpha_m)$.
Suppose now that $\pi_1,\ldots,\pi_i$ are constructed. 
We need to define $\pi_{i+1}$.

Now consider $\alpha_{i+1}$. Let $p_{i+1}$ and $r_{i+1}$ be the minimal polynomials of $\alpha_{i+1}$ over $K$ and $M_i$, respectively.
Note that $r_{i+1}|p_{i+1}$ and $[L_{i+1}:K]=\deg(p_{i+1})=2$.
So can distinguish two cases:
\begin{itemize}
    \item $\deg(r_{i+1})=1$: in this case $M_i = M_{i+1}$, as already $\alpha_{i+1} \in M_i$. So we can define $\pi_{i+1} = \pi_i$;
    \item $\deg(r_{i+1})=2$: in this case we get that $r_{i+1} = p_{i+1}$. We define $\pi_{i+1}$ by setting $\pi_{i+1}|_{M_i}\coloneqq \pi_i$ and $\pi_{i+1}(\alpha_{i+1})\coloneqq \psi_{i+1}(\alpha_{i+1})$.
\end{itemize}
We are left to show that $\pi_{i+1}$ is well defined in the second case. Here, $\alpha_{i+1} \not\in M_i$, but satisfies some quadratic relation $\alpha_{i+1}^2 = u \alpha_{i+1} + t$ with $u,t \in K$, as $r_{i+1}=p_{i+1}$.
Using this, and recalling that $\pi_i|_K=\psi_i|_K=\psi_{i+1}|_K=\phi$  one can see that
$\pi_{i+1}$ is well defined.

By construction, $M_m=K(\alpha_1,\ldots,\alpha_m)$.
Hence we get a homomorphism $\psi\coloneqq \pi_m \colon K(\alpha_1,\ldots,\alpha_m) \to K$ and the commutative diagram,
\[
\begin{tikzcd}
K \ar[rr, "\phi"] \ar[rd, hook] && K \\
& M_m \ar[ru, "\psi=\pi_m"']
\end{tikzcd}
\]
i.e., the extension $K\subset K(\alpha_1,\ldots,\alpha_m)$ is rationalizable.
\end{proof}

The converse statement holds in greater generality.

\begin{proposition}\label{p:compositum}
Let $K$ and $\alpha_1,\ldots,\alpha_m$ be elements of $\overline{K}$;
let $L_i$ be the extension of $K$ obtained by adjoining  $\alpha_i$,
that is, $L_i=K(\alpha_i)$.
If $K(\alpha_1,\ldots,\alpha_m)$ is rationalizable, then so is the set $\{ L_1,\ldots,L_m\}$. 
\end{proposition}

\begin{proof}
Considering the compositions $\iota_i \colon L_i \into K(\alpha_1,\ldots,\alpha_m) \to K$ shows that each $L_i$ is rationalizable.
\end{proof}

Now we are ready to prove our main theorem.

\begin{proof}[Proof of \autoref{t:split}]
The first implication is a special case of \autoref{p:family}.
The other is simply \autoref{p:compositum}
\end{proof}

The following corollary is a direct application of \autoref{t:split} to the study of rationalizability of sets of square roots, the most common in a physical context.

\begin{corollary}\label{c:FamilySquareRoots}
Let $f_1,\ldots,f_m$ be polynomials in $R=\kk [x_1,\ldots,x_n]$.
There exists a polynomial $p\in R[z]$ such that the set of square roots
$\{\sqrt{f_1},\ldots,\sqrt{f_m} \}$ is rationalizable if and only if the variety $V(p)$ defined by $p$ is unirational.
\end{corollary}
\begin{proof}
Set $K=Q=\bC (x_1,\ldots,x_n)$.
Then consider the quadratic field extensions $L_i\coloneqq K(\sqrt{f_i}) \cong K[z]/(z^2-f_i)$ and the compositum field  $L\coloneqq K(\sqrt{f_1},\ldots,\sqrt{f_m})$.
By the primitive element theorem, we can write $L=K(\alpha)$ for some $\alpha\in \overline{K}$.
If $p$ denotes the minimal polynomial of $\alpha$ over $K$, then
$L\cong K[z]/(p)$.
By \autoref{p:EquivFamilies}, the rationalizability of $\{\sqrt{f_1},\ldots,\sqrt{f_m} \}$ is equivalent to the rationalizability of $\{ L_1,\ldots,L_m\}$.
Note that all the extensions $K\subset L_i$ are quadratic.
Then, by \autoref{t:split}, 
the rationalizability $\{ L_1,\ldots,L_m\}$ is equivalent to the rationalizability of $K\subset L\cong K[z]/(p)$ which, in turn, by \autoref{p:EquivSingle}, is equivalent to the unirationality of $V(p)$, concluding the proof.
\end{proof}

\begin{remark}\label{r:Strategy}
\autoref{c:FamilySquareRoots} will prove very useful in the study of rationalizability of families of square roots and represents a more solid approach to the problem than the one given in~\cite[Remark 52]{BF21}.
We remark that usually the polynomial $p$ can be explicitly computed, even by hand.

On the other hand, this result is not very handy when it comes to \emph{disprove} rationalizability of sets of square roots;
for this goal~\cite[Proposition 47]{BF21} stays a more practical tool.
\end{remark}

Looking at its proof and at \autoref{l:tower}, one can propose generalizations of \autoref{t:split}.

\begin{corollary}\label{c:Generalizations}
Let $K$ be a field and
let $\alpha_1,\ldots,\alpha_m$ be  elements of $\overline{K}$.
For $i=1,\ldots,m$, set $M_0\coloneqq K, M_i\coloneqq K(\alpha_1,\ldots,\alpha_i)$,
and $L_i\coloneqq K(\alpha_i)$.
Assume one of the following statement.
\begin{enumerate}
    \item For every $i=1,\ldots,m$, the minimal polynomial of $\alpha_{i}$ over $M_{i-1}$ is either equal to the minimal polynomial of $\alpha_i$ over $K$ or it has degree $1$ (i.e., $\alpha_i\in M_{i-1}$).
    \item For every $i=1,\ldots,m$, the extension $M_{i-1}\subset M_i$ is rationalizable.
\end{enumerate}
Then the set $\{L_1,\ldots,L_m\}$ is rationalizable if and only if $K\subset M_m=K(\alpha_1,\ldots,\alpha_m)$ is.
\end{corollary}
\begin{proof}
If $K\subset M_m$ is rationalizable, then  so is $\{L_1,\ldots,L_m\}$ by \autoref{p:compositum}.

Conversely, assume that $\{L_1,\ldots,L_m\}$ is rationalizable.
If (1) holds, then the proof goes exactly as in the proof of \autoref{p:family}.
If (2) holds, then the statement follows directly from \autoref{l:tower}.
\end{proof}

\section{Applications and further discussions}\label{s:Applications}
In this section we use \autoref{t:split} to deduce some results about rationalizability of families of square roots;
we then apply these new results to some examples taken from~\cite{BF21},
recovering the same conclusions with  shorter arguments.

The computation of a primitive element of a field extension can be done by hand (if the degree is small enough) or with the aid of any computer algebra system.
For some of the computations we used the algebra software \texttt{Magma}~\cite{Magma} and \texttt{Singular}~\cite{DGPS}, as they allow the computation of the geometric genus of a curve.

As before, $\kk$ denotes an algebraically closed field of characteristic $0$ and we define
$R \coloneqq  \kk[x_1,\ldots,x_n]$ and $Q \coloneqq  \Frac(R) = \kk(x_1,\ldots,x_n)$,
with $n\geq 1$.

\begin{corollary}\label{c:TwoLinear}
For any two linear polynomials $f,g\in Q$ the field extension $Q \subset Q (\sqrt{f}, \sqrt{g})$ is rationalizable. 
\end{corollary}
\begin{proof}
By~\autoref{t:split} $Q\subset Q (\sqrt{f}, \sqrt{g})$ is rationalizable if and only if the family of extensions $\{ \frac{Q[z]}{(z^2-f)}, \frac{Q[z]}{(z^2-g)} \}$ is, which in turn,
by~\autoref{p:EquivFamilies}, 
is equivalent to saying that the set of square roots $\{ \sqrt{f}, \sqrt{g} \}$ is rationalizable.
So in order to prove the statement it is enough to prove that $\{\sqrt{f}, \sqrt{g}\}$ is rationalizable.
In order to do so we will use the strategy already explained in~\cite[Remark 52]{BF21}.

First note that without loss of generality, we can always assume $f=x_1$.
This implies that $\sqrt{f}$ is rationalized, for instance, by the map $\phi_1\colon Q\to Q$, sending $x_1\mapsto x_1^2$ and acting as the identity on all the other variables. 
In this way we obtain 
$\phi_1(f)=x^2$ and $\phi_1(g) \eqqcolon g'$ is a polynomial of degree $2$.
Consider the square root $\sqrt{g'}$ and notice that it is rationalizable (cf.~\cite[Theorem 1]{BF21}).
Let $\phi_2\colon Q\to Q$ be the rationalizing homomorphism 
then we have $\phi_2 (g')=h^2$ for some $h\in Q$.
Consider then the composition $\phi\coloneqq \phi_2\circ \phi_1 \colon Q\to Q$ and note that $\phi (f) = (\phi_2(x))^2$ and $\phi (g) = h^2$, proving that $\{\sqrt{f}, \sqrt{g}\}$ is rationalizable.
\end{proof}

\begin{corollary}\label{c:Linearset}
Let $a_1,\ldots,a_m\in \bC$ be pairwise distinct complex numbers. 
The set $\cA \coloneqq  \{ \sqrt{x-a_1},\ldots, \sqrt{x-a_m} \}$ is rationalizable if and only if $m\leq 2$.
\end{corollary}
\begin{proof}
If $m=1$, then $\cA$ is trivially rationalizable (cf.~\cite[Theorem 1]{BF21}).

If $m=2$, then \autoref{c:TwoLinear} implies that $\cA$ is rationalizable.

If $m\geq 3$, then the polynomial $p\coloneqq \prod_{i=1}^m (x-a_i)$ is of degree $m\geq 3$ with $m$ distinct roots.
It follows that the square root $\sqrt{p}$ is not rationalizable (cf.~\cite[Theorem 2]{BF21}) and hence, in turn, $\cA$ is not rationalizable by~\cite[Proposition 47]{BF21}.
\end{proof}

\begin{remark}
We observe that \autoref{c:Linearset} can also be proved using~\autoref{r:GenusHigherOrder}.
\end{remark}

\begin{example}[{c.f. \cite[Example 52]{BF21}}]
The set $\{ \sqrt{x-1}, \sqrt{x-2} \}$
is rationalizable by \autoref{c:Linearset}.
\end{example}

\begin{corollary}\label{c:QuadraticJacobianset}
Let $a,b,c\in \bC$ three pairwise distinct complex numbers. 
Then the set 
$$
\cA\coloneqq \{ \sqrt{(x-a)(x-b)}, \sqrt{(x-a)(x-c)}, \sqrt{(x-b)(x-c)}\}
$$ 
is rationalizable.
\end{corollary}
\begin{proof}
By applying an affine transformation of $\bA^1\subset \bP^1$ we can assume without loss of generality that $a=0$ and $b=1$ and write $c=\lambda$. 
Set 
\begin{align*}
    f_1&\coloneqq (x-a)(x-b)=x(x-1),\\
    f_2&\coloneqq (x-a)(x-c)=x(x-\lambda),\\
    f_3&\coloneqq (x-b)(x-c)=(x-1)(x-\lambda),
\end{align*}
and, for $i=1,2,3$, consider the field extension 
$$
L_i\coloneqq \frac{\bC(x)[z]}{(z^2-f_i)}\cong \bC (x) (\sqrt{f_i}).
$$
Then the rationalizability of $\cA$ is equivalent to the rationalizability of the set $\{ L_1, L_2, L_3\}$.
From~\autoref{t:split} this is equivalent to the rationalizability of the compositum field $L=\bC (x) (\sqrt{f_1}, \sqrt{f_2}, \sqrt{f_3})$.

First notice that, as 
$$
\sqrt{f_3}=\sqrt{(x-1)(x-\lambda)}= \frac{1}{x}\sqrt{x(x-1)}\sqrt{x(x-\lambda)}=\frac{1}{x}\sqrt{f_1}\sqrt{f_2},
$$
the compositum field $L$ is equal to just $\bC (x) (\sqrt{f_1}, \sqrt{f_2})$.
It is then immediate to see that $[L:\bC (x)]=4$.
Let $p\in \bC (x)[z]$ be the irreducible polynomial of the primitive element of $L$ over $\bC (x)$, that is,
$p$ is an irreducible polynomial of degree $4$ such that
$$
L\cong \frac{\bC (x)[z]}{(p)}.
$$
Notice that in general, by possibly multiplying by an element of $\bC(x)$, we can always assume that $p\in \bC [x,z]$.
In particular, in our case
$$
p=z^4 + (-4x + 2\lambda + 2)xz^2 + (\lambda^2 - 2\lambda + 1)x^2.
$$
From~\autoref{p:alggeom} we know that $L$ is rationalizable if and only if $V(p)$ is unirational.
The variety $V(p)\subset \bA^2$ is a curve and
its genus is $0$, proving the statement.
\end{proof}

With the same argument used to prove~\autoref{c:QuadraticJacobianset} one can easily prove also the following result.

\begin{corollary}\label{c:QuadrAlphJacTwoVar}
Let $f$ be a linear polynomial in $\bC (x,y)\setminus \{ x,y \}$.
Then the  following set is rationalizable.
$$
   \cA\coloneqq  \{ \sqrt{xy}, \sqrt{xf}, \sqrt{yf} \}
$$
\end{corollary}

\begin{proof}
Notice that without loss of generality we may (and do) assume that
$f=x+ay+b$ with $a,b\in\bC$.
Using \autoref{t:split} we know that $\cA$ is rationalizable if and only if the extension $L\coloneqq \bC(x,y) (\sqrt{xy}, \sqrt{xf}, \sqrt{yf} ) \supset \bC(x,y)$  is.
Proceeding as in the proof of \autoref{c:QuadraticJacobianset}, one can show that
$$
L\cong \frac{\bC (x,y)[z]}{(p)},
$$
with 
\begin{align*}
    p \, = \,  & x^4 + 2(a - 1)x^3y + 2bx^3 + (a-1)^2x^2y^2 + 2b(a -  1)x^2y + \\
    &  - 2x^2z^2 + b^2x^2  -2(a + 1)xyz^2 - 2bxz^2 + z^4 \; .
\end{align*}

From \autoref{p:alggeom} we know that $L$ is rationalizable if and only if $S=V(p)$ is rational. 
Let $\overline{S} \subset \bP^3$ be the projective closure of $S$.
One can easily verify that $\overline{S}$ is a quartic surface with a singular locus given by the union of two lines $L_1$ and $L_2$ taken with multiplicity $2$; 
the two lines intersect in one point.
Using Urabe's classification of non-normal quartic surfaces (cf.~\cite[\S 1]{Ura86}) we then deduce that $\overline{S}$ is rational, concluding the proof.
\end{proof}

The (rationalizable) situation of Corollaries \ref{c:QuadraticJacobianset} and \ref{c:QuadrAlphJacTwoVar} involving polynomials of degree two is very special in the sense that they give rise to unirational varieties.
In general this does not happen, as the following example shows.

\begin{example}[{c.f. \cite[Example 49]{BF21}}]
In order to study the set
$$
\cA\coloneqq  
 \{
 \sqrt{x},
 \sqrt{1+4x},
 \sqrt{x (x-4)}\}
$$
consider the following family of field extensions of $Q = \kk(x)$:
\[
\{ Q[z]/(z^2-x),\ Q[z]/(z^2-(4x+1)),\ Q[z]/(z^2-x(x-4)) \}.
\]
The compositum of these field extensions is
$Q' = Q[z]/(p)$ with
\[
\begin{split}
p = & z^8 + (-24x+12)z^6 + (144x^2-72x+86)z^4 + \\ 
     & + (-256x^3+96x^2+88x+300)z^2 + 1296x^2+1800x+625
\end{split}
\]
where the roots of $p$ are $\pm \sqrt{x} \pm \sqrt{4x+1} \pm \sqrt{x-4}$.

The geometric genus of the curve $\Spec \kk[x,z]/(p)$ is $1$, hence $Q \subset Q'$ and therefore $\cA$ is not rationalizable.
\end{example}

Notice that Corollaries~\ref{c:TwoLinear}, \ref{c:Linearset},   \ref{c:QuadraticJacobianset} and \ref{c:QuadrAlphJacTwoVar}  are all positive examples of the reverse of \cite[Proposition 47]{BF21},
leading us to formulate \autoref{conj:Sets}.
The next natural step would be to generalize the conjecture to sets of field extensions.
We refrain from doing this at the moment as we do not have strong evidence in this direction.

Another question comes from \autoref{t:split} and \autoref{c:Generalizations}.
If $K$ is any field and $\alpha_1,\alpha_2\in\overline{K}$ are such that $K\subset K(\alpha_i)$ is rationalizable for $i=1,2$,
can we conclude that $K(\alpha_1)\subset K(\alpha_1,\alpha_2)$ is rationalizable?
At the moment the question is wide open.
A positive answer would lead to a stronger form of the main theorem, allowing us to remove the assumption on the degree of the field extensions.

\bibliographystyle{plain}
\bibliography{biblio.bib}

\end{document}